\documentclass[proceedings,submission]{dmtcs}

\usepackage[utf8x]{inputenc}
\usepackage{subfigure}

\usepackage[square,numbers]{natbib}

\usepackage{amssymb}
\usepackage{amsmath}
\usepackage{mathdots}

\usepackage{tikz}
\usetikzlibrary{shapes}
\usetikzlibrary{fit}
\usetikzlibrary{decorations.pathmorphing}

\newtheorem{theorem}{Theorem}[section]
\newtheorem{prop}[theorem]{Proposition}
\newtheorem{lemma}[theorem]{Lemma}
\newtheorem{definition}[theorem]{Definition}

\newtheorem{conj}[theorem]{Conjecture}

\definecolor{Rouge}{RGB}{205,35,38}
\definecolor{Bleu}{RGB}{2,60,195}
\definecolor{vertex}{rgb}{0, 0.5, 0}

\renewcommand{\leq}{\leqslant}
\renewcommand{\geq}{\geqslant}

\author{Philippe Biane\addressmark{1},  Hayat Cheballah\addressmark{1}}
\title[Left  trapezoids]{Gog, Magog and Schützenberger II: left trapezoids}
\address{\addressmark{1}Institut Gaspard-Monge, université Paris-Est Marne-la-Vallée,
5 Boulevard Descartes, Champs-sur-Marne, 77454, Marne-la-Vallée cedex 2,
France}

\begin{document}
\maketitle
\begin{abstract}
\paragraph{Abstract.}
We are interested in finding an explicit bijection between two families of 
combinatorial objects: Gog and Magog triangles. These two families are particular 
classes of  Gelfand-Tsetlin triangles and are respectively in bijection with alternating
sign matrices  (ASM) and totally symmetric self complementary plane partitions (TSSCPP). 
For this purpose, we introduce left Gog and GOGAm trapezoids. We conjecture that these 
two families of trapezoids are equienumerated and we give an explicit bijection between 
the trapezoids with one or two diagonals.

\paragraph{Résumé.}
Nous nous intéressons ici  à trouver une bijection explicite entre deux familles 
d'objets combinatoires: les triangles Gog et Magog. Ces deux familles d'objets sont 
des classes particulières des triangles de Gelfand-Tsetlin et sont respectivement en 
bijection avec les matrices à signes alternants (ASMs) et les partitions planes 
totalement symétriques auto-complémentaires (TSSCPPs). Pour ce faire, nous 
introduisons les Gog et les GOGAm trapèzes gauches. Nous conjecturons que ces deux 
familles de trapèzes sont équipotents et nous donnons une bijection explicite entre 
ces trapèzes à une et deux lignes.
\end{abstract}

\section{Introduction}
\label{sec::intro}
This paper is a sequel to~\cite{BC}, to which we refer for more on the background 
of the Gog-Magog problem (see also~\cite{bressoud} and~\cite{Cheb} for a thorough 
discussion).
It is  a well known open problem in bijective combinatorics  to find a bijection 
between alternating sign matrices and totally symmetric self complementary plane 
partitions. One can reformulate the problem using so-called Gog and Magog triangles, 
which are particular species of Gelfand-Tsetlin triangles. In particular, Gog 
triangles are in simple bijection with alternating sign matrices of the same
size, while Magog triangles are in bijection with totally symmetric 
self complementary plane partitions. In~\cite{mrr},  Mills, Robbins and Rumsey 
introduced trapezoids in this problem by cutting out~$k$ diagonals on the right 
(with the conventions used in the present paper) of a triangle of size~$n$, 
and conjectured that Gog and Magog trapezoids of the same size are equienumerated. 
Zeilberger~\cite{zeilberger} proved this conjecture, but no explicit bijection is known, 
except for~$k=1$ (which is a relatively easy problem) and for~$k=2$, this bijection 
being the main result of~\cite{BC}. In this last paper a new class of triangles and 
trapezoids was introduced, called GOGAm triangles (or trapezoids), which are in 
bijection with the Magog triangles by the Sch\"utzenberger involution acting on 
Gelfand-Tsetlin triangles. 
\smallskip

In this paper we introduce a new class of trapezoids by cutting diagonals of Gog 
and GOGAm triangles on the left instead of the right. We conjecture that the left 
Gog and GOGAm trapezoids of the same shape are equienumerated, and give a bijective 
proof of this for trapezoids composed of one or two diagonals. Furthermore we show 
that our bijection is compatible with the previous bijection between right trapezoids. 
It turns out that the bijection we obtain for left trapezoids is much simpler than the 
one of~\cite{BC} for right trapezoids.Finally we can also consider rectangles 
(intersections of left and right trapezoids). For such rectangles we also conjecture that  
Gog and GOGAm are equienumerated.
\smallskip

Our results are presented in this paper as follows. In section~\ref{sec::basic::def}
we give some elementary definitions about  Gelfand-Tsetlin triangles, Gog and GOGAm triangles,
and then define left and right Gog and GOGAm trapezoids and describe their minimal completion. 
Section~\ref{sec::results::conjecture} is devoted to the formulation of a 
conjecture on the existence of a bijection between Gog and GOGAm trapezoids of the 
same size. We end this paper by section \ref{sec::bijection::gog::gogam::left::trap} 
where we give a bijection between $(n,2)$ left Gog and GOGAm trapezoids and we show 
how its work on an example. Finally, we consider another combinatorial object; rectangles.


\section{Basic definitions}
\label{sec::basic::def}
We start by giving definitions of our main objects of study. We refer to~\cite{BC}
for more details.
\subsection{Gelfand-Tsetlin triangles}
\begin{definition}
A  Gelfand-Tsetlin  triangle of size $n$ is a triangular array 
$X=(X_{i,j})_{n\geqslant i\geqslant j\geqslant 1}$ of positive integers

\begin{equation}
 \begin{array}{ccccccccccccc}
   X_{n,1} &   & X_{n,2}&    &\cdots &        & \cdots        &         &X_{n,n-1}    &      &X_{n,n}\\
    &X_{n-1,1}     &    &X_{n-1,2}     &        &  \cdots  &         &  \iddots  &       &X_{n-1,n-1}     &  \\
    
    &   &\ddots&    & &        &  \iddots   &         &\iddots &  &\\
    &   &    &\ddots &        &\iddots &            & \iddots &        &&&  &   \\
    &   &    &    &X_{2,1}  &        & X_{2,2}   &         &        &&  &&  \\
    &   &    &    &        &X_{1,1}  &            &         &        &    &&  &  
 \end{array}
\end{equation}

 such that
 \begin{equation}
    X_{i+1,j}\leqslant  X_{i,j}\leqslant X_{i+1,j+1}\quad \text{ for } n-1\geqslant i\geqslant j\geqslant 1.
 \end{equation}
\end{definition}

The set of all Gelfand-Tsetlin triangles of size~$n$ is a poset for the order such 
that~$X\leq Y$ if and only if~$X_{ij}\leq Y_{ij}$ for all $i,j$. It is also a lattice 
for this order, the infimum and supremum being taken entrywise:~$\max(X,Y)_{ij}=\max(X_{ij},Y_{ij})$.
\subsection{Gog triangles and trapezoids}
\label{sec.gog}

\begin{definition} A Gog triangle of size~$n$ is a Gelfand-Tsetlin  triangle 
such that
\begin{enumerate}
 \item its rows are strictly increasing;
  \begin{equation}
        X_{i,j}<X_{i,j+1}, \qquad j<i\leqslant n-1
  \end{equation}
 \item and such that 
  \begin{equation}
        X_{n,j}=j,\qquad 1\leqslant j\leqslant n .
  \end{equation}
\end{enumerate}
\end{definition}

Here is an example of Gog triangle of size~$n=5$.
\begin{equation}
    \begin{split}
    \begin{tikzpicture}[scale=.6]
    \draw  (-1,5) node{$1$} ;\draw(1,5) node{$2$} ;\draw(3,5) node{$3$} ;\draw(5,5) node{$4$} ;\draw(7,5) node{$5$} ;        
	    \draw  (0,4) node{1} ;\draw  (2,4) node{3} ;\draw(4,4) node{4} ;\draw(6,4) node{5} ;
		    \draw  (1,3) node{1} ;\draw(3,3) node{4} ;\draw(5,3) node{5} ;
			    \draw  (2,2) node{2} ;\draw(4,2) node{4} ;
				    \draw(3,1) node{3} ;
    \end{tikzpicture}
    \end{split}
\end{equation}

It is immediate to check that the set of Gog triangles of size~$n$
is a sublattice of the Gelfand-Tsetlin triangles.

\begin{definition}
A~$(n,k)$  right Gog trapezoid   (for $k\leq n$) is an array of positive 
integers~$X=(X_{i,j})_{n\geqslant i\geqslant j\geqslant 1; i-j\leq k-1}$ formed 
from the~$k$ rightmost SW-NE diagonals of some Gog triangle of size $n$.
\end{definition}
Below is a $(5,2)$ right Gog trapezoid.

\begin{equation}
    \begin{split}
    \begin{tikzpicture}[scale=.6]
\draw(6,4) node{4} ;\draw(8,4) node{5} ;
\draw(5,3) node{4} ;\draw(7,3) node{5} ;
\draw(4,2) node{3} ;\draw(6,2) node{4} ;
\draw(3,1) node{1} ;\draw(5,1) node{3} ;
\draw(4,0) node{2} ;
    \end{tikzpicture}
    \end{split}
\end{equation}

\begin{definition}
A~$(n,k)$ left Gog  trapezoid  (for~$k\leq n$) is an array of positive 
integers~$X=(X_{i,j})_{n\geqslant i\geqslant j\geqslant 1; k\geqslant j}$  
formed from the~$k$ leftmost NW-SE diagonals of a Gog triangle of size~$n$.
\end{definition}
A more direct way of checking that a left Gelfand-Tsetlin 
trapezoid is a left Gog trapezoid is to verify that its 
rows are strictly increasing and that its SW-NE diagonals 
are bounded by~$1,2,\ldots,n$ as it is shown in the figure below
which represents a~$(5,2)$ left Gog trapezoid.
\smallskip

\begin{equation}\label{eq::52::left::gog::trapezoid}
    \begin{split}
    \begin{tikzpicture}[scale=.6]
         \tikzstyle{inf}=[left,rotate=45,color=blue]
         \tikzstyle{lim}=[color=blue]
\draw  (0,4) node{1} ;\draw  (2,4) node{2} ;
\draw  (1,3) node{1} ;\draw(3,3) node{3} ;
\draw  (2,2) node{2} ;\draw(4,2) node{3} ;
\draw(3,1) node{2} ;\draw(5,1) node{4} ;
\draw(4,0) node{4} ;
\node[inf](d1) at (1,4.75) {$\leqslant$};\node[lim](ld1) at (1.5,5) {$1$};
\node[inf](d2) at (3,4.75) {$\leqslant$};\node[lim](ld2) at (3.5,5) {$2$};
\node[inf](d3) at (4,3.75) {$\leqslant$};\node[lim](ld3) at (4.5,4) {$3$};
\node[inf](d4) at (5,2.75) {$\leqslant$};\node[lim](ld4) at (5.5,3) {$4$};
\node[inf](d5) at (6,1.75) {$\leqslant$};\node[lim](ld5) at (6.5,2) {$5$};
    \end{tikzpicture}
    \end{split}
\end{equation}

There is a simple involution~$X\to\tilde X$ on Gog triangles given by 
\begin{equation}
\tilde X_{i,j}=n+1-X_{i,i+1-j}
\end{equation}
which exchanges left and right trapezoids of the same size.
This involution corresponds to a vertical symmetry of associated ASMs.

\subsubsection{Minimal completion}
Since the set of Gog triangles is a lattice, given a left (resp. a right) 
Gog trapezoid, there exists a smallest Gog triangle from which it 
can be extracted. We call this Gog triangle the {\sl canonical completion} of the 
left (resp. the right) Gog trapezoid. 
Their explicit value is computed in the next Proposition.
\begin{prop}$ $
  \begin{enumerate}
    \item Let~$X$ be a~$(n,k)$ right Gog trapezoid, then its canonical completion satisfies
	\begin{equation}
	      X_{ij}=j\quad \mbox{for} \quad i\geqslant j+k.
	\end{equation}
    \item Let~$X$ be a~$(n,k)$ left Gog trapezoid, then its canonical completion satisfies
	\begin{equation}
	      X_{i,j}=\max(X_{i,k}+j-k,X_{i-1,k}+j-k-1,\ldots, X_{i-j+k,k})\quad \mbox{for} \quad j\geqslant k.
	\end{equation}
  \end{enumerate}
\end{prop}

\begin{proof} The first case  (right trapezoids) is trivial, the formula for the 
second case (left trapezoids) is easily proved  by induction on $j-k$.
\end{proof}

\smallskip

For example, the completion of the $(5,2)$ left Gog trapezoid 
in~\eqref{eq::52::left::gog::trapezoid}
\begin{equation}
    \begin{split}
    \begin{tikzpicture}[scale=.6]
\tikzstyle{ntrap}=[color=red]
\draw  (0,4) node{1} ;\draw  (2,4) node{2} ;\draw[ntrap](4,4) node{3} ;
\draw[ntrap](6,4) node{4} ;\draw[ntrap](8,4) node{5} ;
\draw  (1,3) node{1} ;\draw(3,3) node{3} ;
\draw[ntrap](5,3) node{4} ;\draw[ntrap](7,3) node{5} ;
\draw  (2,2) node{2} ;\draw(4,2) node{3} ;
\draw[ntrap](6,2) node{4} ;
\draw(3,1) node{2} ;\draw(5,1) node{4} ;
\draw(4,0) node{4} ;
    \end{tikzpicture}
    \end{split}
\end{equation}
Remark that the supplementary entries of the canonical completion of a 
left Gog trapezoid depend only on its rightmost NW-SE diagonal. 
\smallskip

The right trapezoids defined above coincide (modulo easy reindexations) 
with those of Mills, Robbins, Rumsey~\cite{mrr}, and Zeilberger~\cite{zeilberger}. 
They are in obvious bijection with the ones in~\cite{BC} 
(actually the Gog trapezoids of~\cite{BC} are the canonical completions 
of the right Gog trapezoids defined above). 

\subsection{GOGAm triangles and trapezoids} 

\begin{definition}
A GOGAm triangle of size~$n$ is a Gelfand-Tsetlin triangle such that~$X_{nn}\leq n$ 
and, for all~$1\leq k\leq n-1$, and all~$n=j_0> j_1> j_2\ldots>j_{n-k}\geq 1$, 
one has
  \begin{equation}\label{GOGAm}
   \left(\sum_{i=0}^{n-k-1}X_{j_i+i, j_i}-X_{j_{i+1}+i,j_{i+1}}\right)+X_{j_{n-k}+n-k,j_{n-k}}\leq k
  \end{equation}
\end{definition}

It is shown in~\cite{BC} that GOGAm triangles are exactly the Gelfand-Tsetlin 
triangles obtained by applying the Sch\"utzenberger involution to Magog triangles. 
It follows that the problem of finding an explicit bijection between Gog and Magog 
triangles can be reduced to that of finding an explicit bijection between Gog and 
GOGAm triangles. In the sequel,  Magog triangles  will  not be considered  anymore.

\begin{definition}
A~$(n,k)$ right GOGAm    trapezoid   (for~$k\leq n$) is an array of positive 
integers\\ $X=(x_{i,j})_{n\geqslant i\geqslant j\geqslant 1; i-j\leq k-1}$ formed 
from the~$k$ rightmost SW-NE diagonals of a GOGAm triangle of size $n$.
\end{definition}
Below is a~$(5,2)$ right GOGAm trapezoid.

\begin{equation}
    \begin{split}
    \begin{tikzpicture}[scale=.6]
\draw(6,4) node{2} ;\draw(8,4) node{4} ;
\draw(5,3) node{2} ;\draw(7,3) node{4} ;
\draw(4,2) node{2} ;\draw(6,2) node{4} ;
\draw(3,1) node{1} ;\draw(5,1) node{4} ;
\draw(4,0) node{3} ;
    \end{tikzpicture}
    \end{split}
\end{equation}

\begin{definition}
A~$(n,k)$ left GOGAm trapezoid (for $k\leq n$) is  an array of positive integers\\
$X=(x_{i,j})_{n\geqslant i\geqslant j\geqslant 1; k\geqslant j}$ formed from the~$k$ 
leftmost NW-SE diagonals of a GOGAm trapezoid of size~$n$.
\end{definition}

Below is a~$(5,2)$ left GOGAm trapezoid.

\begin{equation}\label{eq::52::left::GOGAm::trapezoid}
    \begin{split}
    \begin{tikzpicture}[scale=.6]
\draw  (0,4) node{1} ;\draw  (2,4) node{1} ;
\draw  (1,3) node{1} ;\draw(3,3) node{2} ;
\draw  (2,2) node{1} ;
\draw(4,2) node{2} ;
\draw(3,1) node{2} ;\draw(5,1) node{3} ;
\draw(4,0) node{3} ;
    \end{tikzpicture}
    \end{split}
\end{equation}
\subsubsection{Minimal completion}
The set of GOGAm triangles is not a sublattice of the Gelfand-Tsetlin 
triangles, nevertheless, given a right (resp. a left) GOGAm trapezoid, 
we shall see that there exists a smallest GOGAm triangle which extends 
it. We call this GOGAm triangle the {\sl canonical completion} of the 
left (resp. the right) GOGAm trapezoid. 
\begin{prop}\label{comp}$ $
\begin{enumerate}
 \item Let~$X$ be a~$(n,k)$ right GOGAm trapezoid, 
 then its canonical completion is given by
  \begin{equation}
  X_{ij}=1\quad \text{for $n\geqslant i\geqslant j+k$}.
  \end{equation}
 \item Let~$X$ be a~$(n,k)$ left GOGAm trapezoid, 
 then its canonical completion is given by
  \begin{equation}
  X_{i,j}=X_{i-j+k,k}\quad \text{for $n\geqslant i\geqslant j\geqslant k$}
  \end{equation}
  in other words, the  added entries are  constant on SW-NE diagonals
\end{enumerate}
\end{prop}

\begin{proof}
In both cases, the completion above is the smallest Gelfand-Tsetlin triangle 
containing the trapezoid, therefore it is enough to check that if~$X$ is 
a~$(n,k)$ right or left GOGAm trapezoid, then its completion, 
as indicated in the proposition~\ref{comp},  
is a GOGAm triangle. The claim follows from the following lemma. 

\begin{lemma} Let~$X$ be a GOGAm triangle.
\begin{enumerate}
 \item[{\bf i)}]The triangle obtained from~$X$ by replacing the entries on 
 the upper left triangle~$(X_{ij}, n\geqslant i\geqslant j+k)$ by~$1$ is a GOGAm triangle. 
 \item[{\bf ii)}] Let~$n\geqslant m\geqslant k\geqslant1$.
If~$X$ is constant on each partial SW-NE diagonal~$(X_{i+l,k+l}; n-i\geqslant l\geqslant 0)$ 
for~$i\geqslant m+1$ then the triangle  obtained from~$X$ by 
replacing the entries~$(X_{m+l,k+l};  n-m\geqslant l\geqslant 1)$ 
by~$X_{m,k}$ is a GOGAm triangle.
\end{enumerate}
\end{lemma}

\begin{proof}
It is easily seen that the above replacements give a Gelfand-Tsetlin triangle.
Both proofs then follow  by inspection of the formula~(\ref{GOGAm}), 
which shows   that, upon making the above replacements, the quantity on the 
left cannot increase. 
 \end{proof}

{\sl End of proof of Proposition~\ref{comp}.}
The case of right GOGAm triangles is dealt with by part~{\bf i)} of the 
preceding Lemma.
The case of left trapezoids follows by replacing successively 
the SW-NE partial diagonals as in part {\bf ii)} of the Lemma. 
\end{proof}
\smallskip

For example, the completion of the $(5,2)$ left GOGAm trapezoid 
in~\eqref{eq::52::left::GOGAm::trapezoid} is as follows.
\begin{equation}
    \begin{split}
    \begin{tikzpicture}[scale=.6]
\tikzstyle{ntrap}=[color=red]
\draw  (0,4) node{1} ;\draw  (2,4) node{1} ;
\draw[ntrap](4,4) node{2} ;
\draw[ntrap](6,4) node{2} ;\draw[ntrap](8,4) node{3} ;
\draw  (1,3) node{1} ;\draw(3,3) node{2} ;
\draw[ntrap](5,3) node{2} ;\draw[ntrap](7,3) node{3} ;
\draw  (2,2) node{1} ;
\draw(4,2) node{2} ;
\draw[ntrap](6,2) node{3} ;
\draw(3,1) node{2} ;\draw(5,1) node{3} ;
\draw(4,0) node{3} ;
    \end{tikzpicture}
    \end{split}
\end{equation}

\section{Results and conjectures}
\label{sec::results::conjecture}
\begin{theorem}[Zeilberger \cite{zeilberger}]
For all~$k\leq n$, the~$(n,k)$ right Gog and GOGAm   
trapezoids are equienumerated
\end{theorem}

Actually Zeilberger proves this theorem for Gog and Magog 
trapezoids, but composing by the Sch\"utzenberger involution yields the
above result.
In \cite{BC} a bijective proof is given for $(n,1)$ and $(n,2)$ right trapezoids.

\begin{conj}\label{conjecture::equipotence::gog::GOGAm::nk::gauche}
For all~$k\leq n$, the~$(n,k)$ left Gog and GOGAm trapezoids are equienumerated.
\end{conj}

In the next section we will give a bijective proof 
of this conjecture for~$(n,1)$ and~$(n,2)$ trapezoids.

Remark that the right and left Gog trapezoids of   shape~$(n,k)$ 
are equienumerated (in fact a simple bijection between 
them was given above).
\smallskip

If we consider left GOGAm  trapezoids as GOGAm triangles, 
using the canonical completion, then we can take their image by the 
Sch\"utzenberger involution and obtain a subset of the Magog triangles, 
for each~$(n,k)$. It seems however that this subset does not have a simple 
direct characterization. This shows that GOGAm triangles and trapezoids are 
a useful tool in  the bijection problem between Gog and Magog triangles.


\section{Bijections between Gog and GOGAm left trapezoids}
\label{sec::bijection::gog::gogam::left::trap}
\subsection{$(n,1)$ left trapezoids}
The sets of  $(n,1)$  left Gog trapezoids and of  $(n,1)$  left GOGAm 
trapezoids  coincide with the set of sequences~$X_{n,1},\ldots, X_{1,1}$ 
satisfying~$X_{j,1}\leq n-j+1$ (note that these sets  are counted by Catalan numbers). 
Therefore the identity map provides a trivial bijection between these two sets.

\subsection{$(n,2)$ left trapezoids}
In order to treat the~$(n,2)$ left  trapezoids we will recall some definitions 
from~\cite{BC}.
\subsubsection{Inversions}  

\begin{definition}
An  inversion in a Gelfand-Tsetlin triangle is 
a pair $(i,j)$ such that $X_{i,j}=X_{i+1,j}$.
\end{definition}

For example, the  Gog triangle in~\eqref{eq::inversion::triangle::gog} 
contains three inversions, $(2,2)$, $(3,1)$, $(4,1)$, the respective 
equalities being in red 
on this picture.
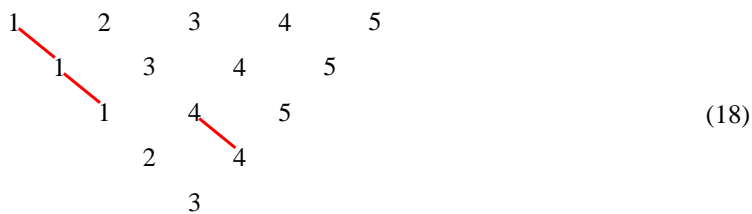
\begin{equation}\label{eq::inversion::triangle::gog}
    \begin{split}
    \begin{tikzpicture}[scale=.6]
\tikzstyle{inv}=[color=red,line width=.4mm]
       \draw  (1,3) node{1} ;\draw  (3,3) node{2} ;\draw  (5,3) node{3} ;\draw  (7,3) node{4} ;\draw  (9,3) node{5} ;
             \draw  (2,2) node{1} ;\draw  (4,2) node{3} ;\draw  (6,2) node{4} ;\draw  (8,2) node{5} ;
                   \draw  (3,1) node{1} ;\draw  (5,1) node{4} ;\draw  (7,1) node{5} ;
                         \draw  (4,0) node{2} ;\draw  (6,0) node{4} ;
                         \draw  (5,-1) node{3} ;
  \draw[inv] (1.9,2.2)--(1.1,2.85);\draw[inv] (2.9,1.2)--(2.1,1.85);
  \draw[inv] (5.9,.2)--(5.1,.85);
  \end{tikzpicture}
    \end{split}
 \end{equation}

\begin{definition}
Let~$X=(X_{i,j})_{n\geqslant i\geqslant j\geqslant 1}$ be a Gog triangle 
and let~$(i,j)$ be such that~$1\leqslant j\leqslant i\leqslant n$. 

An inversion~$(k,l)$ covers~$(i,j)$ if~$i=k+p$ and~$j=l+p$ for 
some~$p$ with~$1\leqslant p\leqslant n-k$. 
\end{definition}
 The entries~$(i,j)$ covered by an inversion are depicted with~$"+"$ on 
 the following picture.

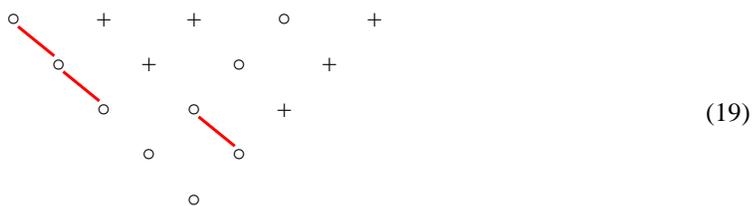
\begin{equation}\label{eq::couverture::triangle::generique}
    \begin{split}
    \begin{tikzpicture}[scale=.6]
\tikzstyle{inv}=[color=red,line width=.4mm]
      \draw  (1,3) node{$\circ$} ;\draw  (3,3) node{+} ;\draw  (5,3) node{+} ;\draw  (7,3) node{$\circ$} ;\draw  (9,3) node{+} ;
            \draw  (2,2) node{$\circ$} ;\draw  (4,2) node{+} ;\draw  (6,2) node{$\circ$} ;\draw  (8,2) node{+} ;
                  \draw  (3,1) node{$\circ$} ;\draw  (5,1) node{$\circ$} ;\draw  (7,1) node{+} ;
                        \draw  (4,0) node{$\circ$} ;\draw  (6,0) node{$\circ$} ;
                                 \draw  (5,-1) node{$\circ$} ;
  \draw[inv] (1.9,2.2)--(1.1,2.85);\draw[inv] (2.9,1.2)--(2.1,1.85);
  \draw[inv] (5.9,.2)--(5.1,.85);
  \end{tikzpicture}
    \end{split}
 \end{equation}

\subsubsection{Standard procedure}
The basic idea for our bijection  is that for any inversion 
in the Gog triangle we should subtract $1$  to the entries covered by 
this inversion, scanning the inversions along the successive NW-SE diagonals, 
starting from the rightmost diagonal, and scanning each diagonals from NW to SE. 
We call this the {\sl standard procedure}. This procedure does not always yield 
a Gelfand-Tsetlin triangle, but one can check that it does so if one starts from 
a Gog triangle corresponding to a permutation matrix in the correspondance 
between alternating sign matrices and Gog triangles. Actually the triangle 
obtained is also a GOGAm triangle.

Although we will not use it below, it is informative to make the following remark.

\begin{prop}
Let~$X$ be the  canonical completion of a left~$(n,k)$ Gog trapezoid. 
The triangle~$Y$ obtained by applying  the standard procedure to 
the~$n-k+1$ rightmost NW-SE diagonals of~$X$ is a Gelfand-Tsetlin triangle such 
that~$Y_{i+l,k+l}=X_{i,k}$ for~$n-i\geqslant l\geqslant 1$.
\end{prop}
\begin{proof}
The Proposition is proved easily by induction on the number~$n-k+1$. 
\end{proof}

For example, applied to the $(5,2)$ left Gog trapezoid in \eqref{eq::52::left::gog::trapezoid}, this yields
\smallskip

\begin{equation}
    \begin{split}
    \begin{tikzpicture}[scale=.6]
\tikzstyle{ntrap}=[color=red]
\draw  (0,4) node{$1$} ;\draw  (2,4) node{$2$} ;
\draw[ntrap](4,4) node{$3$} ;
\draw[ntrap](6,4) node{$3$} ;\draw[ntrap](8,4) node{$4$} ;
\draw  (1,3) node{$1$} ;\draw(3,3) node{$3$} ;
\draw[ntrap](5,3) node{$3$} ;\draw[ntrap](7,3) node{$4$} ;
\draw  (2,2) node{$2$} ;\draw(4,2) node{$3$} ;
\draw[ntrap](6,2) node{$4$} ;
\draw(3,1) node{$2$} ;\draw(5,1) node{$4$} ;
\draw(4,0) node{$4$} ;
    \end{tikzpicture}
    \end{split}
\end{equation}

Like in~\cite{BC} the bijection between left Gog and GOGAm 
trapezoids will be obtained by a modification of the Standard Procedure.

\subsubsection{Characterization of $(n,2)$ GOGAm trapezoids}
The family of  inequalities~\eqref{GOGAm} simplifies in the case 
of~$(n,2)$ GOGAm trapezoids, indeed if we identify such a trapezoid 
with its canonical completion, then most of the terms in the left 
hand side are zero, so that these inequalities reduce to

\begin{eqnarray}
X_{i,2}&\leqslant& n-i+2\label{GOGAm21}\\
X_{i,2}-X_{i-1,1}+X_{i,1}&\leqslant &n-i+1\label{GOGAm22}
\end{eqnarray}

Remark that,  since~$-X_{i-1,1}+X_{i,1}\leqslant 0$, the inequality \eqref{GOGAm22} follows 
 from \eqref{GOGAm21} unless~$X_{i-1,1}=X_{i,1}$.
\subsubsection{From Gog to GOGAm}
Let~$X$ be a~$(n,2)$ left Gog trapezoid. We shall construct 
a~$(n,2)$ left GOGAm trapezoid~$Y$ by scanning the inversions in the 
leftmost NW-SE diagonal of~$X$, starting from NW. Let us denote 
by~$n> i_1>\ldots>i_k\geq 1$ these inversions, so that~$X_{i,1}=X_{i+1,1}$ 
if and only if~$i\in\{i_1,\ldots,i_k\}$. We also put~$i_0=n$. 
We will  construct  a sequence of~$(n,2)$ left Gelfand-Tsetlin
trapezoids~$X=Y^{(0)},Y^{(1)},Y^{(2)},\ldots,Y^{(k)}=Y$. 
 \smallskip

Let us  assume that we have constructed the trapezoids up to~$Y^{(l)}$,   
that~$Y^{(l)}\leq X$,  that~$Y^{(l)}_{ij}= X_{ij} $  for~$i\leqslant i_{l} $,
and that inequalities \eqref{GOGAm21} and \eqref{GOGAm22} are satisfied 
by~$Y^{(l)}$ for~$i\geqslant i_l+1$. This is the case for~$l=0$.
\smallskip

Let~$m$ be the largest integer such that~$Y^{(l)}_{m,2}=Y^{(l)}_{i_{l+1}+1,2}$.
We put
$$\begin{array}{rclcl}
Y^{(l+1)}_{i,1}&=&Y^{(l)}_{i,1}&\quad\mbox{ for }\quad& n\geqslant i\geqslant m\ \quad\mbox{ and }\quad  i_{l+1}\geqslant i\\
Y^{(l+1)}_{i,1}&=&Y^{(l)}_{i+1, 1}&\quad\mbox{ for }\quad& m-1\geqslant i> i_{l+1}\\  
Y^{(l+1)}_{i,2}&=&Y^{(l)}_{i,2}&\quad\mbox{ for }\quad& n\geqslant i\geqslant m+1\ \quad\mbox{ and }\quad  i_{l+1}\geqslant i\\
Y^{(l+1)}_{i,2}&=&Y^{(l)}_{i,2}-1&\quad\mbox{ for }\quad& m\geqslant i\geqslant i_{l+1}+1.                 
\end{array}$$
\smallskip

From the definition of~$m$, and the fact that~$X$ is a Gog trapezoid, 
we see that this new triangle is a Gelfand-Tsetlin triangle, 
that~$Y^{(l+1)}\leq X$, and that~$Y^{(l+1)}_{ij}= X_{ij} $ for~$i\leqslant i_{l+1}$.
Let us now check that the  trapezoid~$Y^{(l+1)}$ satisfies the 
inequalities~\eqref{GOGAm21} and \eqref{GOGAm22} for~$i\geqslant i_{l+1}+1$.
The first series of inequalities, for~$i\geqslant i_{l+1}+1$, 
follow from the fact that~$Y^{(l)}\leq X$.
For the second series, they are satisfied for~$i\geqslant m+1$ 
since this is the case for~$Y^{(l)}$. For~$m\geqslant i\geqslant i_{l+1}+1$, 
observe that 
\begin{equation*}
 Y^{(l+1)}_{i,2}-Y^{(l+1)}_{i+1,1}+Y^{(l+1)}_{i+1,1}\leqslant Y^{(l+1)}_{i,2}= Y^{(l+1)}_{m,2}=Y^{(k)}_{m,2}-1\leqslant n-m+1
\end{equation*}
 by \eqref{GOGAm21} for~$Y(l)$, from which \eqref{GOGAm22} follows.
\smallskip
 
This proves that~$Y^{(l+1)} $ again satisfies the induction hypothesis. 
Finally~$Y=Y^{(k)}$ is a GOGAm triangle: indeed inequalities \eqref{GOGAm21}
follow again from~$Y^{(l+1)}\leqslant X$, and \eqref{GOGAm22} for~$i\leq i_{k}$ 
follow from the fact that there are no inversions in this range. It follows that
the above algorithm provides a map from~$(n,2)$ left Gog trapezoids to~$(n,2)$ 
left  GOGAm trapezoids.Observe that the number of inversions in the leftmost 
diagonal of~$Y$ is the same as for~$X$, but the positions of these inversions 
are not the same in general.
\subsubsection{Inverse map}\label{invert}

We now describe the inverse map, from GOGAm left trapezoids to Gog left trapezoids.
\smallskip

We start from an~$(n,2)$ GOGAm left trapezoid~$Y$, and construct a sequence \\
$Y=Y^{(k)},Y^{(k-1)},Y^{(k-2)},\ldots, Y^{(0)}=X$ of intermediate Gelfand-Tsetlin 
trapezoids.\\ Let~$n-1\geqslant \iota_1>\iota_2\ldots>\iota_k\geqslant 1$ be 
the inversions of the leftmost diagonal of~$Y$, and let~$\iota_{k+1}=0$.
 Assume that~$Y^{(l)}$ has been constructed and that~$Y_{ij}^{(l)}=Y_{ij}$ 
 for~$i-j\geqslant \iota_{l+1}$. 
This is the case for~$l=k$.
\smallskip

Let~$p$ be the smallest integer such that~$Y^{(l)}_{i_{l}+1,2}=Y^{(l)}_{p,2}$. 
We put 
$$\begin{array}{rclcl}
 Y^{(l-1)}_{i,1}&=&Y^{(l)}_{i,1}&\quad\mbox{ for }\quad& n\geqslant i\geqslant \iota_{l}+1\ \quad\mbox{ and }\quad p\geqslant i\\
Y^{(l-1)}_{i,1}&=&Y^{(l)}_{i-1,1}&\quad\mbox{ for }\quad& \iota_l\geqslant i\geqslant p\\  
Y^{(l-1)}_{i,2}&=&Y^{(l)}_{i,2}&\quad\mbox{ for }\quad& n\geqslant i\geqslant \iota_l+2\ \quad\mbox{ and }\quad p-1\geqslant i\\
Y^{(l-1)}_{i,2}&=&X^{(l)}_{i,2}-1&\quad\mbox{ for }\quad& \iota_l+1\geqslant i\geqslant p.
\end{array}$$
\smallskip

It is immediate to check that if~$X$ is an~$(n,2)$ left Gog trapezoid, 
and~$Y$ is its image by the first algorithm then the above algorithm 
applied to~$Y$ yields~$X$ back, actually the sequence~$Y^{(l)}$ is the same. 
Therefore in order to prove the bijection we only need to show that 
if~$Y$ is a~$(n,2)$ left GOGAm trapezoid then the algorithm is well defined 
and~$X$ is a Gog left trapezoid. This is a bit cumbersome, but not difficult, 
and very similar to the opposite case, so we leave this task to the reader.

\subsubsection{A statistic}\label{stats}
Observe that in our bijection the value of the bottom 
entry~$X_{1,1}$ is unchanged when we go from Gog to GOGAm trapezoids.
The same was true of the bijection in~\cite{BC} for right trapezoids. 
Actually we make the following conjecture, which extends 
Conjecture~\ref{conjecture::equipotence::gog::GOGAm::nk::gauche} above.
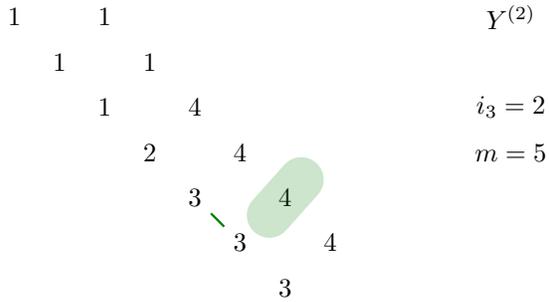
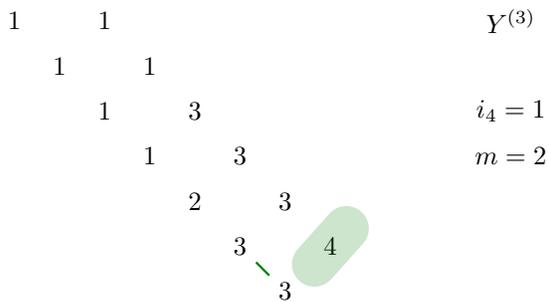
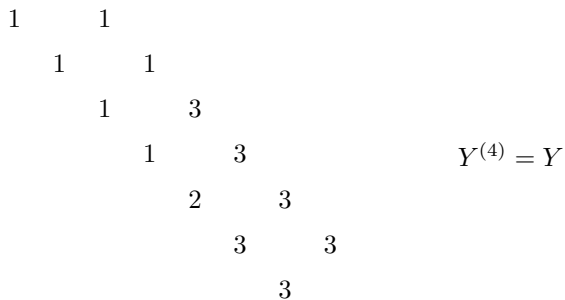
\begin{conj}
For each~$n,k,l$ the~$(n,k)$ left Gog and GOGAm 
trapezoids with bottom entry~$X_{1,1}=l$ are equienumerated.
\end{conj}
\subsection{An example}
In this section we work out an example of the algorithm from the Gog trapezoid~$X$  
to the  GOGAm trapezoid~$Y$ by showing the successive trapezoids~$Y^{(k)}$.
At each step we indicate the  inversion in green, as well as the entry 
covered by this inversion in shaded green, and the values of the parameters~$i_l, p$.
The algorithm also runs backwards to yield the GOGAm$\to$Gog bijection.
\smallskip

\begin{tikzpicture}[scale=0.6]
\tikzstyle{barr}= [opacity=.2,line width=6 mm,cap=round,color=red]
\tikzstyle{barb}= [opacity=.2,line width=6 mm,cap=round,color=blue]
\tikzstyle{barv}= [opacity=.2,line width=6 mm,cap=round,color=violet]
\tikzstyle{barvr}= [opacity=.2,line width=6 mm,cap=round,color=vertex]
\tikzstyle{comp}= [color=vertex]
\tikzstyle{ntrap}=[color=red]
\tikzstyle{inv}=[color=vertex,line width=0.3mm]
\begin{scope}
\node (11)at (1,10){$1$};\node (21)at (2,9) {$1$} ;\node  (31) at (3,8){$1$} ;
\node(41)at  (4,7) {$2$} ;\node (51)at  (5,6) {$3$} ;
\node (61)at  (6,5) {$3$} ;
\node (71)at  (7,4) {$3$} ;
\node (12)at  (3,10) {$2$} ;
\node (22)at  (4,9) {$2$} ;
\node (32)at (5,8){$4$};
\node (42)at (6,7){$4$};
\node(52)at  (7,6) {$4$} ;
\node(62)at  (8,5) {$4$} ;
 \draw[inv](11)--(21);
\draw[barvr](12.south west)--(12.north east);
\node at (12,10){$X=Y^{(0)}$};
\node at (12,8){$i_1=6$};
\node at (12,7){$m=7$};
\end{scope}
\begin{scope}[yshift=-8cm]
\node (11)at (1,10){$1$};\node (21)at (2,9) {$1$} ;\node  (31) at (3,8){$1$} ;
\node(41)at  (4,7) {$2$} ;\node (51)at  (5,6) {$3$} ;
\node (61)at  (6,5) {$3$} ;
\node (71)at  (7,4) {$3$} ;
\node (12)at  (3,10) {$1$} ;
\node (22)at  (4,9) {$2$} ;
\node (32)at (5,8){$4$};
\node (42)at (6,7){$4$};
\node(52)at  (7,6) {$4$} ;
\node(62)at  (8,5) {$4$} ;

 \draw[inv](21)--(31);
\draw[barvr](22.south west)--(22.north east);
\node at (12,10){$Y^{(2)}$};
\node at (12,8){$i_2=5$};
\node at (12,7){$m=6$};
\end{scope}
 \end{tikzpicture}

\begin{tikzpicture}[scale=.6]
 \tikzstyle{barr}= [opacity=.2,line width=6 mm,cap=round,color=red]
\tikzstyle{barb}= [opacity=.2,line width=6 mm,cap=round,color=blue]
\tikzstyle{barv}= [opacity=.2,line width=6 mm,cap=round,color=violet]
\tikzstyle{barvr}= [opacity=.2,line width=6 mm,cap=round,color=vertex]
\tikzstyle{comp}= [color=vertex]
\tikzstyle{ntrap}=[color=red]
\tikzstyle{inv}=[color=vertex,line width=0.3mm]
\begin{scope}
\node (11)at (1,10){$1$};\node (21)at (2,9) {$1$} ;\node  (31) at (3,8){$1$} ;
\node(41)at  (4,7) {$2$} ;\node(51)at  (5,6) {$3$} ;
\node (61)at  (6,5) {$3$} ;
\node (71)at  (7,4) {$3$} ;
\node (12)at  (3,10) {$1$} ;
\node (22)at  (4,9) {$1$} ;
\node (32)at (5,8){$4$};
\node (42)at (6,7){$4$};
\node(52)at  (7,6) {$4$} ;
\node(62)at  (8,5) {$4$} ;

 \draw[inv](51)--(61);
\draw[barvr](52.south west)--(52.north east);
\node at (12,10){$Y^{(2)}$};
\node at (12,8){$i_3=2$};
\node at (12,7){$m=5$};
\end{scope}
\begin{scope}[yshift=-8cm]
\node (11)at (1,10){$1$};\node (21)at (2,9) {$1$} ;\node  (31) at (3,8){$1$} ;
\node(41)at  (4,7) {$1$} ;\node (51)at  (5,6) {$2$} ;
\node (61)at  (6,5) {$3$} ;
\node (71)at  (7,4) {$3$} ;
\node (12)at  (3,10) {$1$} ;
\node (22)at  (4,9) {$1$} ;
\node (32)at (5,8){$3$};
\node (42)at (6,7){$3$};
\node(52)at  (7,6) {$3$} ;
\node(62)at  (8,5) {$4$} ;

 \draw[inv](61)--(71);
\draw[barvr](62.south west)--(62.north east);

\node at (12,10){$Y^{(3)}$};
\node at (12,8){$i_4=1$};
\node at (12,7){$m=2$};
\end{scope}
\begin{scope}[yshift=-16cm]
\node (11)at (1,10){$1$};\node (21)at (2,9) {$1$} ;\node  (31) at (3,8){$1$} ;
\node(41)at  (4,7) {$1$} ;\node (51)at  (5,6) {$2$} ;
\node (61)at  (6,5) {$3$} ;
\node (71)at  (7,4) {$3$} ;
\node (12)at  (3,10) {$1$} ;
\node (22)at  (4,9) {$1$} ;
\node (32)at (5,8){$3$};
\node (42)at (6,7){$3$};
\node(52)at  (7,6) {$3$} ;
\node(62)at  (8,5) {$3$} ;

\node at (12,7){$Y^{(4)}=Y$};
\end{scope}
\end{tikzpicture}

\subsection{Rectangles}
\begin{definition}
For~$(n,k,l)$ satisfying~$k+l\leqslant n+1$, a~$(n,k,l)$ 
Gog rectangle is an array of positive 
integers~$X=(x_{i,j})_{n\geqslant i\geqslant j\geqslant 1; k\geqslant j;j+l\geqslant i+1}$
formed from the intersection of the~$k$ leftmost NW-SE diagonals 
and the~$l$  rightmost SW-NE  diagonals of a Gog triangle of size $n$.
\end{definition}

\begin{definition}
For~$(n,k,l)$ satisfying~$k+l\leqslant n+1$, a~$(n,k,l)$ GOGAm 
rectangle is  an array of positive 
integers~$X=(x_{i,j})_{n\geqslant i\geqslant j\geqslant 1; k\geqslant j;j+l\geqslant i+1}$
formed from the intersection of the~$k$ leftmost NW-SE diagonals 
and the~$l$  rightmost SW-NE  diagonals of a GOGAm triangle of size~$n$.
\end{definition}
Similarly to the case of trapezoids, one can check that 
a~$(n,k,l)$ Gog (resp. GOGAm) rectangle has a canonical (i.e. minimal) 
completion as  a~$(n,k)$ left  trapezoid, or as a~$(n,l)$ right trapezoid, 
and finally as a  triangle of size~$n$.
\begin{conj}
For any~$(n,k,l)$ satisfying~$k+l\leqslant n+1$ the~$(n,k,l)$ 
Gog and GOGAm rectangles are equienumerated.
\end{conj}

As in the case of trapezoids, there is also a refined version 
of the conjecture with the statistic~$X_{11}$ preserved.

One can check, using standard completions,  that our bijections, 
in~\cite{BC} and in the present paper, restrict to bijections for 
rectangles of size~$(n,k,2)$ or~$(n,2,l)$. Furthermore, 
in the case of~$(n,2,2)$ rectangles, the bijections coming from left 
and right trapezoids are the same.

\acknowledgements
\label{sec:ack}
This work is based on computer exploration and the authors use the open-source 
mathematical software Sage~\cite{sage} and one of its 
extention, Sage-Combinat~\cite{sagecomb}

\bibliographystyle{abbrvnat}
\bibliography{fpsacbib}
\label{sec:biblio}

\end{document}